\documentclass[10pt, reqno]{amsart}
\usepackage{amsmath,amsfonts,amsthm,bm} 

\usepackage{amsfonts}
\usepackage{amssymb}
\usepackage{latexsym}

\usepackage{tikz}
\usepackage{graphicx,graphics,epsfig}
\usepackage{epstopdf}	
\usepackage[noadjust]{cite}

\usepackage[]{enumerate}
\usepackage{verbatim}
\usepackage{bbm}

\usepackage[]{enumerate}

\usepackage[sc]{mathpazo}          
\usepackage{eulervm}               
\usepackage[scaled=0.86]{berasans} 
\usepackage[scaled=1]{inconsolata} 
\usepackage[T1]{fontenc}

\usepackage[final]{hyperref}
\hypersetup{colorlinks=true, linkcolor=blue, anchorcolor=blue, citecolor=red, filecolor=blue, menucolor=blue, pagecolor=blue, urlcolor=blue}

\newcommand{\Bigabs}[1]{\Bigl\vert #1 \Bigr\vert}

\newcommand{\norm}[1]{\left\Vert #1 \right\Vert}

\newcommand{\R}{\mathbb{R}}

\newcommand{\angles}[1]{\langle #1 \rangle}

\newtheorem{theorem}{Theorem}

\newtheorem{lemma}{Lemma}

\newtheorem*{PWtheorem}{Paley-Wiener Theorem}

\theoremstyle{definition}

\theoremstyle{remark}

\title[Radius of analyticity for fifth order KdV-BBM Equation ]{Lower bound on the radius of analyticity of solution for fifth order KdV-BBM Equation}

  \author[B. Belayneh]{ Birilew Belayneh }
\author[E. Tegegn]{Emawayish Tegegn}

\address{Department of Mathematics
\\
Bahir Dar University 
\\
Ethiopia}
 
   \email{birilewb@yahoo.com, emaway93tegegn@yahoo.com}

\author[A. Tesfahun]{Achenef Tesfahun}

\address{Department of Mathematics \\
Nazarbayev University \\
Qabanbai Batyr Avenue 53 \\
010000 Nur-Sultan \\
Republic of Kazakhstan}

\email{achenef@gmail.com}


\begin{document}

    \begin{abstract}
    We show that the uniform radius of spatial analyticity $\sigma(t)$ of solution at time $t$ for the fifth order KdV-BBM equation cannot decay faster than $1/t$ for large $t>0$, given initial data that is analytic with fixed radius $\sigma_0$. This significantly improves a recent result by Carvajal and 
    Panthee, where they established an exponential decay of $\sigma(t)$ for large $t$.
\end{abstract}

\maketitle

\section{Introduction}
We consider the fifth order KdV-BBM type equation of the form
\begin{equation}\label{bbm}
   \begin{split}
          &\partial_t \eta +\partial_x \eta  -\gamma_1 \partial_t \partial_x^2 \eta + \gamma_2\partial_x^3 \eta+ \delta_1 \partial_t \partial_x^4 \eta+ \delta_2 \partial_x^5 \eta
          \\
         & \qquad \qquad =- \frac 34 \partial_x (\eta^2)- \gamma \partial^3_x (\eta^2) + \frac 7{48} \partial_x (   \eta_x^2)+ \frac18  \partial_x (\eta^3),
             \end{split}
\end{equation}
where the unknown function is
$$ \eta: \R^{1+1} \rightarrow \R.$$ The parameters
$\gamma_1, \gamma_2,  \delta_1, \delta_2, \gamma $ are constants that satisfy certain constraints;  see \cite{BCPS2018, CP2020} for more details.

The fifth-order PDE \eqref{bbm} describes the unidirectional propagation of water waves, and this was
recently introduced by Bona et al. \cite{BCPS2018} by using the second order approximation in the two-way model, the so-called \emph{abcd-system} derived in \cite{BCS2002, BCS2004}.

We compliment \eqref{bbm} with initial data
\begin{equation}\label{data}
    \eta(x, 0)=\eta_0(x).
\end{equation}
In the case $\gamma=  7/48$, the energy
\[
\mathcal E[\eta(t)] =\frac12 \int_{\R}\left( \eta^2 + \gamma_1 \eta_x^2 + \delta_1 \eta_{xx}^2\right) \ dx
\]
is conserved by the flow of \eqref{bbm}, i.e., 
\begin{equation}\label{E-conv}
\mathcal E[\eta(t)] = \mathcal E[\eta_0] \quad \text{for all } \ t.
\end{equation}

Local well-posedness of the Cauchy problem \eqref{bbm}-\eqref{data} with data in the Sobolev spaces $H^s(\R)$ for $s\ge 1$ was established by Bona et al in \cite{BCPS2018}.  When $\gamma_1, \delta_1 > 0$ and $\gamma=  7/48$, the authors \cite{BCPS2018} used the energy conservation \eqref{E-conv} to prove global well-posedness of  \eqref{bbm}-\eqref{data} 
 for data in $H^s(\R)$, $s\ge 2$. Furthermore, the authors used the method of \emph{high-low frequency splitting} to obtain global well-posedness for data with Sobolev regularity $3/2 \le s < 2$. 
This global well-posedness result was further improved in \cite{CP2019} for initial data with Sobolev regularity $s\ge 1$.

Recently, Carvajal and Panthee \cite{CP2020} studied the property of spatial analyticity of the solution $\eta(x, t)$ to \eqref{bbm}-\eqref{data}, given that the initial data $\eta_0(x)$ 
is real-analytic
with uniform radius of analyticity $\sigma_0$, so there is a holomorphic extension to a complex strip
\[
    S_{\sigma_0} = \{ x + iy \in \mathbb{C}:\ |y| < \sigma_0 \}.
\]
The authors proved that, for short times, the radius of analyticity $\sigma(t)$ of the solution $\eta(x,t)$ remains at least as large as the initial radius, i.e. one can take $\sigma(t)=\sigma_0$. On the other hand, for large times, they proved that $\sigma(t)$ decays exponentially in $t$.
In the present paper, we use the idea introduced in \cite{ST2015} (see also \cite{SD2015, AT2017}) to improve this result significantly showing $\sigma(t)$ cannot decay faster than $1/t$ for large $t$. 

For earlier studies
concerning properties of spatial analyticity of solutions for a large class of nonlinear partial differential equations, see for instance \cite{BGK2005, BGK2006}, \cite{Ferrari1998}--\cite{KM1986}, \cite{Levermore1997}--\cite{AT2019}.

A class of analytic function spaces suitable to study analyticity of solution is the  Gevrey class, denoted $G^{\sigma, s}(\R)$, which are defined by the norm
\[
\| f \|_{G^{\sigma, s}} =\|e^{\sigma |D_x|} \angles{D_x}^s f\|_{L_x^2},
\]
where $D_x=-i\partial_x$ and $\angles{\cdot }= \sqrt{1+ |\cdot|^2}$.

The reason for considering initial data in the space $G^{\sigma, s}$ is due to the analyticity properties of Gevrey functions, which are detailed in the following theorem:
\begin{PWtheorem}
Let $\sigma > 0$ and $s \in \mathbb{R}$. Then the following are equivalent:
\begin{enumerate}[(a)]
    \item  $f \in G^{\sigma, s}(\R)$,
    \item $f$ is the restriction to $\R$ of a function $F$ which is holomorphic in the strip
\[
    S_{\sigma} = \{ x + iy \in \mathbb{C}:\ |y| < \sigma \}.
\]
\end{enumerate}
Moreover, the function $F$ satisfies the estimates
\[
    \sup_{|y| < \sigma} \| F (\cdot + i y) \|_{H^{s}} < \infty.
\]
\end{PWtheorem}

\noindent A proof can be found in \cite{K1976} in the case $s = 0$; the general case follows from a simple modification.  The quantity $\sigma$ is known as the \emph{radius of analyticity}.

We remark that Gevrey spaces satisfy the embeddings
\begin{equation}\label{embed0}
\| f \|_{G^{\sigma, s}} \le C \| f \|_{G^{\sigma', s'}}
\end{equation}
for any $s, s' \in \mathbb{R}$, $\sigma <\sigma'$ and some constant $C>0$. This implies 
\begin{equation}\label{embed1}
\| f \|_{H^{s}} \le C \| f \|_{G^{\sigma', s'}}
\end{equation}
for all $s,s' \in \R$, $\sigma'>0$. 

As a consequence of property \eqref{embed1} and the existing well-posedness theory in $H^s(\R)$ (see \cite{BCPS2018}), we conclude that the Cauchy problem \eqref{bbm}-\eqref{data} (with $\gamma_1, \delta_1 > 0$ and $\gamma= 7/48$) has a unique, smooth solution for all time, given initial data 
$\eta_0 \in G^{\sigma_0, s}$ for all $\sigma_0>0$ and $s\in \R$.

Our main result is as follows:  
\begin{theorem}\label{thm-gwp}
Assume $\gamma_1, \delta_1 > 0$ and $\gamma= 7/48$.  Suppose that
 $\eta$ is the global solution of \eqref{bbm}--\eqref{data} 
with $\eta_0 \in G^{\sigma_0, 2}$ for $\sigma_0 > 0$. Then 
\[
 \eta(t) \in G^{\sigma(t), 2} \quad \text{for all } \ t>0,
\]
with the radius of analyticity $\sigma(t)$ satisfying the asymptotic lower bound
\[
\sigma(t) \ge  \frac c t \quad \text {as} \ t \rightarrow  + \infty,
\]
where $c > 0$ is a constant depending on 
 the initial data norm $\norm{\eta_0}_{G^{\sigma_0, 2}}$.
\end{theorem}
Thus, the solution at any time $t$, $\eta(t)$, is analytic in the strip $S_{\sigma(t)}$.

\medskip
\noindent \textbf{Notation}. 
For any positive numbers $a$ and $b$, the notation $a\lesssim b$ stands for $a\le cb$, where $c$ is a positive constant that may change from line to line. Moreover, we denote $a \sim b$  when  $a \lesssim b$ and $b \lesssim a$.

\section{Local well-posedness in $G^{\sigma_0, s}$}
We outline the argument in \cite{CP2020} that enables the authors to obtain the local well-posedness result for \eqref{bbm}-\eqref{data} in $G^{\sigma_0, s}$ with $s\ge 1$, $\sigma_0 > 0$.

Taking the spatial Fourier transform of \eqref{bbm} we obtain
\begin{equation}\label{bbmFT}
    i \partial_t \widehat{\eta}  - \phi(\xi)  \widehat{\eta}= \tau (\xi)   \widehat{\eta^2}-\frac18 \psi(\xi)  \widehat{\eta^3} -\frac 7{48} \psi(\xi)  \widehat{\eta_x^2}, 
\end{equation}
where 
\begin{align*}
    \phi(\xi) = \frac{\xi\left( 1- \gamma_2 \xi^2 + \delta_2 \xi^4\right)}{\varphi(\xi)}, \quad \psi(\xi) = \frac{\xi}{\varphi(\xi)}, \quad \tau(\xi) = \frac{\xi\left( 3- 4 \gamma \xi^2 \right)}{4\varphi(\xi)}
\end{align*}
with 
\begin{align*}
    \varphi(\xi) = 1+ \gamma_1 \xi^2 + \delta_1 \xi^4.
\end{align*}

Defining the Fourier multipliers
\begin{align*}
    \phi(D_x) f =\mathcal F^{-1}\left[ \varphi(\xi) f \right], \quad \psi(D_x) f =\mathcal F^{-1}\left[ \psi(\xi) f \right], \quad \tau(D_x) f =\mathcal F^{-1}\left[ \tau(\xi) f\right],
\end{align*}
we can rewrite \eqref{bbmFT} in an operator form as
\begin{equation}\label{bbmFM}
    i \partial_t \eta  - \phi(D_x)  \eta= F(\eta),
\end{equation} 
where 
\begin{equation}\label{bbmNL}
    F(\eta)= \tau (D_x)   \eta^2-\frac18 \psi(D_x)  \eta^3 -\frac 7{48} \psi(D_x) \eta_x^2.
\end{equation} 

Then the integral equation for \eqref{bbmFM}--\eqref{bbmNL} with initial data \eqref{data} is given by 
\begin{equation}\label{bbmIntE}
   \eta(t)= e^{-it \phi(D_x)} \eta_0 - i \int_0^t e^{-i(t-t') \phi(D_x)}  F(\eta) (t') \, dt'.
\end{equation}

Combining the estimates in \cite[Lemma 2.2--2.4]{CP2020}, we obtain the following nonlinear estimate.
\begin{lemma}\label{lm-nlest}\cite[Lemma 2.2--2.4]{CP2020}
Let $F(\eta)$ be defined as in \eqref{bbmNL}. Then for $s\ge 1$, $\sigma > 0$, we have
\begin{equation}\label{lmnlest}
    \norm{ F(\eta)}_{G^{\sigma, s}} \lesssim \left[1+\norm{\eta}_{G^{\sigma, s}} \right] \norm{\eta}^2_{G^{\sigma, s}} 
\end{equation}
for all $\eta \in G^{\sigma, s}$.

\end{lemma}

Now applying the contraction mapping argument to the integral equation \eqref{bbmIntE} and using Lemma \ref{lm-nlest} yields the following local result.
\begin{theorem}\label{thm-lwp}\cite{CP2020}
Let $s\ge 1$, $\sigma_0 > 0$ and $\eta_0 \in G^{\sigma_0, s}$. Then there exists a unique solution  
\[
 \eta \in C \left( [0, T];  G^{\sigma_0, s} (\R) \right) 
\]
of the Cauchy problem \eqref{bbm}-\eqref{data}, where the existence time is 
\begin{equation}\label{T}
 T\sim  \left(1+ \norm{\eta_0}\right)^{-2}_{G^{\sigma_0, s}}. 
 \end{equation}
Moreover, 
\begin{equation}\label{LocSoln-bound}
    \norm{\eta}_{L^\infty_T G^{\sigma_0, s}} \lesssim \norm{\eta_0}_{G^{\sigma_0, s}}.
\end{equation}
\end{theorem}

Here we use the notation
$$
 L_T^\infty G^{\sigma_0, s} =  L_t^\infty G^{\sigma_0, s} \left( [0, T] \times \R \right).
$$

\section{Almost conservation law and proof of Theorem \ref{thm-gwp}}\label{cons}

We fix $\gamma_1, \delta_1 > 0$ and $\gamma= 7/48$. Let
\[
v(x,t): = \Lambda_\sigma  \eta(x,t), \quad \text{where} \quad \Lambda_\sigma =e^{\sigma|D_x|}  .\]
Then $\eta= \Lambda_{-\sigma} v $. Note also that
$v_0:=v(x,0)=\Lambda_\sigma  \eta_0 $.

Now define the modified energy
\[
\mathcal E_{\sigma}[v(t)] =\frac12 \int_{\R}\left( v^2 + \gamma_1 v_x^2 + \delta_1 v_{xx}^2\right) \ dx.
\]
Observe that that for $\sigma=0$, we have $v=\eta$, and therefore the energy conserved, i.e.,
$
\mathcal E_0[v(t)]= \mathcal E_0[v_0] 
$
for all $t$.
However, this fails to hold for $\sigma>0$.   In what follows we will nevertheless prove the approximate conservation
\begin{equation*}\label{approx2}
\sup_{0 \le t\le T} \mathcal E_\sigma[v(t)] =  \mathcal E_\sigma[v_0] +  \sigma  \cdot  \mathcal O \left(    \left[ 1+ \mathcal (E_{\sigma}[v_0])^\frac12\right] (\mathcal E_{\sigma}[v_0])^\frac 32 \right)
\end{equation*} 
 for $T$ as in Theorem \ref{thm-lwp}.
Thus, in the limit as $\sigma \rightarrow 0$, we recover the conservation $\mathcal E_0[v(t)]= \mathcal E_0[v_0] $.

\subsection{Almost conservation law}

Applying the operator $\Lambda_\sigma $ to equation \eqref{bbm} we obtain
\begin{equation}\label{bbm-md}
\begin{split}
    & \partial_t v +\partial_x v -\gamma_1 \partial_t \partial_x^2 v+ \gamma_2\partial_x^3 v + \delta_1 \partial_t \partial_x^4 v+ \delta_2 \partial_x^5 v 
    \\
    & \qquad \qquad =- \left( \frac 34 
    + \gamma \partial^2_x \right) \partial_x (v^2)  +   \gamma \partial_x (   v_x^2)+ \frac18  \partial_x (v^3) + N(v),
    \end{split}
\end{equation}
where 
\begin{equation}\label{N}
 N(v) = \left(\frac 34  +\gamma \partial^2_x \right) \partial_x N_1(v)
 - \gamma \partial_x N_2(v) - \frac18  \partial_x N_3(v)
\end{equation}
with 
\begin{equation}\label{Nj}
\begin{split}
     N_1(v) &= v^2   - \Lambda_\sigma\left[ ( \Lambda_{-\sigma} v)^2 \right], \\
 N_2(v) &= v_x^2   - \Lambda_\sigma\left[ ( \Lambda_{-\sigma} v_x)^2 \right], 
 \\
 N_3(v) &= v^3   - \Lambda_\sigma\left[ ( \Lambda_{-\sigma} v)^3 \right].
 \end{split}
\end{equation}

Using integration by parts\footnote{Assuming that the solution is sufficiently regular.} and \eqref{bbm-md}--\eqref{Nj} we compute
\begin{align*}
\frac{d}{dt}\mathcal E_{\sigma}[v(t) ] &= \int_{\R}\left( vv_t + \gamma_1 v_x v_{xt} + \delta_1 v_{xx} v_{xxt} \right) \ dx
\\
&= \int_{\R} v\left( \partial_t v - \gamma_1 \partial_t \partial_x^2 v + \delta_1\partial_t \partial_x^4 v  \right) \ dx
\\ 
&= -\int_{\R} v\left( \partial_x v + \gamma_2  \partial_x^3 v + \delta_2  \partial_x^5 v  + \frac 34 \partial_x (v^2)+ \gamma \partial^3_x (v^2) - \gamma \partial_x (   v_x^2)- \frac18  \partial_x (v^3)  \right) \ dx 
\\
& \qquad   \qquad  \qquad 
 + \int_{\R} v N(v)  \ dx.
\end{align*}
The integral on the third line is zero due to the identities
\begin{align*}
    v \partial_x v &= \frac12 \left(v^2\right)_x, \qquad  v \partial^3_x v = \left(v v_{xx}\right)_x  -\frac12 \left(v_x^2\right)_x,
    \\
  v \partial^5_x v &= \left(v \partial^4_{x} v\right)_x - \left(\partial_x v \partial^3_{x} v\right)_x  + \frac12 \left(v_{xx}^2\right)_x
\end{align*}
and 
\begin{align*}
   v \partial_x \left(v^2 \right)&=  \frac23 \left(v^3 \right)_x, \qquad v \partial_x \left(v^3 \right)= \frac34 \left(v^4 \right)_x  ,\\
  v \partial^3_x (v^2)&=2 \left(v^2 v_{xx} \right)_x + v \left(v_x^2 \right)_x.
\end{align*}

Therefore,
\begin{align*}
\frac{d}{dt}\mathcal E_{\sigma}[v(t)] &=  \int_{\R} v N(v)  \ dx.
\end{align*}
Consequently, 
\begin{equation}\label{EnergyEst}
\begin{split}
\mathcal E_{\sigma}[v(t) ] &= \mathcal E_{\sigma}[v(0)] + \int_{0}^{t} \frac{d}{ds} \mathcal E_{\sigma} [v(s)] \ ds
\\
&= \mathcal E_{\sigma}[v_0]  +   \int_{0}^{t} \int_{\R} v(x,s)  N\left( v(x, s) \right) \, dx ds.
\end{split}
\end{equation}

Now we state a key estimate that will be proved in the last section.
\begin{lemma}\label{lm-error-est}
We have
\begin{equation}\label{ErrorEst}
 \Bigabs{\int_{\R} v  N\left( v\right) \, dx }     \le C \sigma  \left[ 1+\norm{
v }_{  H^2}\right] \norm{
v }^3_{ H^2}
\end{equation}
for all $v\in  H^2 $.
\end{lemma}

So in view of \eqref{EnergyEst} and \eqref{ErrorEst}, we have the a priori energy estimate
\begin{equation}\label{approx1}
\sup_{0 \le t\le T} \mathcal E_{\sigma}[v(t)] 
=  \mathcal E_{\sigma}[v_0] + \sigma T \cdot \mathcal O \left( \left[ 1+\norm{
v }_{ L_T^\infty H^2}\right] \norm{
v }^3_{ L_T^\infty H^2} \right),
\end{equation} 
where 
$$
 L_T^\infty H^2: =  L_t^\infty H^2([0, T]\times \R).
$$
We combine this estimate with the local existence theory in Theorem \ref{thm-lwp} above to obtain an almost conservation law to the modified energy.
\begin{lemma}\label{lm-approx}[Almost conservation law]
Let $\eta_0 \in G^{\sigma, 2}$. Suppose that $\eta \in C\left( [0, T]; G^{\sigma, 2}\right)$ is the local-in-time solution to the Cauchy problem \eqref{bbm}-\eqref{data} that is constructed in Theorem \ref{thm-lwp}. Then
\begin{equation}\label{approx2}
\sup_{0 \le t\le T} \mathcal E_\sigma[v(t)] =  \mathcal E_\sigma[v_0] +  \sigma  \cdot  \mathcal O \left(    \left[ 1+ \mathcal (E_{\sigma}[v_0])^\frac12\right] (\mathcal E_{\sigma}[v_0])^\frac 32 \right).
\end{equation} 

\end{lemma}

\begin{proof}
By Theorem \ref{thm-lwp} we have the bound 
\begin{equation}\label{v-databd}
 \norm{v}_{ L_T^\infty H^2}=  \norm{\eta  }_{L_T^\infty G^{\sigma, 2}} \le C \norm{\eta_0 }_{  G^{\sigma, 2}} = C \norm{v_0 }_{  H^2} ,
\end{equation}
where $T$ is as in \eqref{T}.
On the other hand, for fixed constants $\gamma_1, \delta_1 > 0$, we have 
\begin{equation}\label{En-data}
\begin{split}
\mathcal E_{\sigma_0}[v_0] &=\frac12 \int_{\R}\left( v_0^2 + \gamma_1 ( v'_0)^2 + \delta_1 ( v''_0)^2\right) \ dx  
\\
&\sim \norm{v_0}^2_{H^2} .
\end{split}
\end{equation}
Then combining \eqref{v-databd}--\eqref{En-data} with 
\eqref{approx1} yields the desired estimate \eqref{approx2}.
\end{proof}

 \subsection{Proof of Theorem \ref{thm-gwp}}
Suppose that $\eta_0 \in G^{\sigma_0, 2} $ for some $\sigma_0 > 0$.
From the local theory there is 
a unique solution 
$$
\eta \in C\left( [0, T]; G^{\sigma_0, 2}(\R)\right) 
$$
of \eqref{bbm}, \eqref{data} constructed in Theorem \ref{thm-lwp} with existence time 
$T$ as in \eqref{T}.

Note that since 
$$ 
v_0= e^{\sigma_0 |D_x|} \eta_0 \in H^{ 2} 
$$
we have 
\begin{align*}
\mathcal E_{\sigma_0}[v_0] \sim  \norm{v_0}^2_{H^2} <\infty.
\end{align*}
Now following the argument in \cite{SD2015, AT2017}
we can construct a solution on $[0, T_\ast]$ for arbitrarily large time $T_\ast$ by applying the approximate conservation law in Lemma \ref{lm-approx}, \eqref{approx2}, so as to repeat the above local result on successive short time intervals of size $T$ to reach $T_\ast$ by adjusting the strip width parameter $\sigma$ according to the size of $T_\ast$.  Doing so, we establish the bound
\begin{equation}\label{keybound}
\sup_{t\in [0, T_\ast]}  \mathcal E_\sigma[v(t)] \le 2 \mathcal E_{\sigma_0}[v_0] 
\end{equation}
for $\sigma$ satisfying 
\begin{equation}
\label{siglb}
\sigma(t) \ge C/T_\ast .
\end{equation}
Thus, $\mathcal E_\sigma(t) < \infty$ for $t \in [0,T_\ast]$, which in turn implies 
\[
 \eta (t) \in G^{\sigma(t), 2} \quad \text{for all } \ t\in [0, T_\ast].
\]

\section{Proof of Lemma \ref{lm-error-est}}
Estimate  \eqref{ErrorEst} reduces to 
\begin{equation}\label{NLest}
    \Bigabs{ \int_{\R} v N\left(v \right) \, dx} \le C\sigma \left[ 1+\norm{
v }_{H^2}\right] \norm{
v}^3_{H^2},
\end{equation}
where $v$ can be regarded as a function of only $x$ (since $t$ is fixed).

Using \eqref{N}-\eqref{Nj}, Plancherel and Cauchy-Schwarz, we get
\begin{align*}
    \int_{\R} v N\left(v \right) \, dx
&=  \int_{\R} v \left( \frac34  + \gamma \partial^2_x \right)\partial_x N_1\left(v \right) \, dx  - \gamma   \int_{\R} v \partial_x N_2\left(v \right) \, dx - \frac18   \int_{\R} v \partial_x N_3\left(v \right) \, dx
\\
&= \int_{\R} \left( \frac34  + \gamma \partial^2_x \right) v \cdot  \partial_x N_1\left(v \right) \, dx  + \gamma   \int_{\R} \partial_x v \cdot  N_2\left(v \right) \, dx  + \frac18    \int_{\R} \partial_x v \cdot  N_3\left(v \right) \, dx
\\ 
&\le \norm{ \left( \frac34  + \gamma \partial^2_x \right) v}_{L^2_x} \norm{\partial_x N_1\left(v \right)}_{L^2_x} + \gamma \norm{ \partial_x  v}_{L^2_x} \norm{ N_2\left(v \right)}_{L^2_x} + \frac18 \norm{ \partial_x  v}_{L^2_x} \norm{ N_3\left(v \right)}_{L^2_x}
\\ 
&\le \norm{v}_{H^2} \norm{\partial_x N_1\left(v \right)}_{L^2_x} + \gamma \norm{   v}_{H^1} \norm{ N_2\left(v \right)}_{L^2_x} + \frac18 \norm{  v}_{H^1} \norm{ N_3\left(v \right)}_{L^2_x}.
\end{align*}
So \eqref{NLest} follows from the following estimates:
\begin{align}
\label{N1}
  \norm{\partial_x N_1\left(v \right)}_{L^2_x} \lesssim  \sigma  \norm{v}^2_{H^2}  
  \\
  \label{N2}
  \norm{ N_2\left(v \right)}_{L^2_x} \lesssim \sigma  \norm{v}^2_{H^2}  
  \\
  \label{N3}
  \norm{ N_3\left(v \right)}_{L^2_x} \lesssim \sigma  \norm{v}^3_{H^2} .
\end{align}

\subsection{Proof of \eqref{N1}}
By taking the FT, we write 
 \begin{align*}
\widehat{ \partial_x N_1(v)}(\xi) & =  i\int_{ \xi = \xi_{1} + \xi_{2}} \xi \left(e^{ \sigma (|\xi_1|+ |\xi_2|)} - e^{\sigma|\xi|} \right) \hat{\eta}(\xi_{1})  \hat{\eta}(\xi_{2}) \ d\xi_1  d \xi_{2} \\
& =i \int_{ \xi = \xi_{1} + \xi_{2}}  \xi p_\sigma(\xi_1, \xi_2)) \hat{v}(\xi_{1})  \hat{v}(\xi_{2}) \ d\xi_1  d \xi_{2},
\end{align*}
where
\begin{align*}
    p_\sigma(\xi_1, \xi_2)) & =1 - \exp\left(-\sigma \left[  (|\xi_1|+ |\xi_2|) -\left| \xi_1+ \xi_2  \right| \right]\right).
\end{align*}

Since $1 - e^{-r} \leq r$ for all $r\ge 0$, we have
\begin{equation}\begin{aligned}\label{pest}
|p_\sigma(\xi_1, \xi_2))| & \le \sigma \left[ (|\xi_1|+ |\xi_2|) -\left| \xi_1+ \xi_2  \right|  \right] \\
&= \sigma \frac{ \left( |\xi_1|+ |\xi_2|\right)^2 -\left|  \xi_1+ \xi_2  \right|^2 }{ |\xi_1|+ |\xi_2| +\left| \xi_1+ \xi_2  \right|   }
\\
& \leq 2 \sigma \min \left( |\xi_1|,  |\xi_2| \right)
\end{aligned}\end{equation}

By symmetry, we may assume $|\xi_1| \le |\xi_2|$. This implies $|\xi|\le 2|\xi_2|$. Now 
let $$ V= \mathcal F_x^{-1} \left( |\widehat{v}|\right).$$
Then by \eqref{pest}
\begin{align*}
|\widehat{\partial_x N_1(v)}(\xi)|& \leq 4 \sigma \int_{\xi = \xi_{1} + \xi_{2}}   |\xi_1||\widehat v(\xi_{1}) |\cdot |\xi_2 ||\widehat{v}(\xi_{2}) |\, d\xi_1 d\xi_2
\\
&=4 \sigma \int_{\xi = \xi_{1} + \xi_{2}}    |\xi_1|\widehat V(\xi_{1}) \cdot |\xi_2 | \widehat{V}(\xi_{2})  \, d\xi_1 d\xi_2
\\
&=4 \sigma \mathcal F_x\left[ |D_x|V \cdot  |D_x|V  \right](\xi).
\end{align*}
Therefore, using Plancherel, H\"{o}lder and Sobolev inequalities we get
\begin{align*}
\| \partial_x N_1(U) \|_{L^{2}_x} & \le 4 \sigma \| |D_x|V \cdot  |D_x|V  \|_{L^{2}_{x}}
\\
&\le  4 \sigma \| |D_x|V  \|_{L^{2}_{x}} \||D_x|V   \|_{L^{\infty}_{x}}
\\
 & \lesssim  \sigma \|  V \|_{H^{2}}^2 \sim  \sigma \|  v \|_{H^{2}}^2
\end{align*}
as desired.

\subsection{Proof of \eqref{N2}}
By taking the FT, we write 
\begin{align*}
\widehat{N_2(v)}(\xi) & =  \int_{ \xi = \xi_{1} + \xi_{2}} \left(e^{ \sigma (|\xi_1|+ |\xi_2|)} - e^{\sigma|\xi|} \right)  \widehat{\eta_x}(\xi_{1})  \widehat{\eta_x}(\xi_{2}) \ d\xi_1  d \xi_{2} \\
& =- \int_{ \xi = \xi_{1} + \xi_{2}} \xi_1 \xi_2 p_\sigma(\xi_1, \xi_2) \hat{v}(\xi_{1})  \hat{v}(\xi_{2}) \ d\xi_1  d \xi_{2},
\end{align*}
where $p_\sigma(\xi_1, \xi_2)|$ as in the preceding subsection.

Assuming $|\xi_1| \le |\xi_2|$, by symmetry, we have by \eqref{pest}
\begin{align*}
    |\xi_1 \xi_2 p_\sigma(\xi_1, \xi_2)| & \le 2\sigma  |\xi_1|^2  |\xi_2|.
\end{align*}
Then 
\begin{align*}
|\widehat{ N_2(v)}(\xi)|& \leq 2 \sigma \int_{\xi = \xi_{1} + \xi_{2}}   |\xi_1|^2 |\widehat v(\xi_{1}) |\cdot |\xi_2 ||\widehat{v}(\xi_{2}) |\, d\xi_1 d\xi_2
\\
&=2 \sigma \int_{\xi = \xi_{1} + \xi_{2}}    |\xi_1|^2 \widehat V(\xi_{1}) \cdot |\xi_2 | \widehat{V}(\xi_{2})  \, d\xi_1 d\xi_2
\\
&=2 \sigma \mathcal F_x\left[ |D_x|^2 V \cdot  |D_x|V  \right](\xi).
\end{align*}
Therefore, by Plancherel, H\"{o}lder and Sobolev inequalities we get
\begin{align*}
\|  N_2(U) \|_{L^{2}_x} & \le 2 \sigma \| |D_x|^2 V \cdot  |D_x|V \|_{L^{2}_{x}}  
\\
&\le   2 \sigma \| |D_x|^2 V  \|_{L^{2}_{x}} \||D_x|V   \|_{L^{\infty}_{x}} 
\\
 & \lesssim \sigma \|  V \|_{H^{2}}^2 \sim 2 \sigma \|  v \|_{H^{2}}^2
\end{align*}
as desired.

\subsection{Proof of \eqref{N3}}

\begin{align*}
\widehat{N_3(v)}(\xi) & =  \int_{ \xi = \xi_{1} + \xi_{2} + \xi_3}\left(e^{\sum_{j=1}^{3} \sigma|\xi_j|} - e^{\sigma|\xi|} \right) \hat{\eta}(\xi_{1})  \hat{\eta}(\xi_{2}) \hat{\eta}(\xi_{3})\ d\xi_1  d \xi_{2} d \xi_{3} \\
& = \int_{ \xi = \xi_{1} + \xi_{2}+ \xi_{3}}  q_\sigma(\xi_1, \xi_2, \xi_3)) \hat{v}(\xi_{1})  \hat{v}(\xi_{2}) \hat{v}(\xi_{3}) \ d\xi_1  d \xi_{2} d\xi_3,
\end{align*}
where
\begin{align*}
     q_\sigma(\xi_1, \xi_2, \xi_3)) & =1 - \exp\left(-\sigma \left[ \sum_{j = 1}^{3}|\xi_j| -\left|\sum_{j = 1}^{3} \xi_j  \right| \right]\right).
\end{align*}
We estimate
\begin{equation}\begin{aligned}\label{qest}
|q_\sigma(\xi_1, \xi_2, \xi_3))| & \le \sigma \left[ \sum_{j = 1}^{3}|\xi_j| -\left|\sum_{j = 1}^{3} \xi_j  \right| \right] \\
&= \sigma \frac{ \left( \sum_{j = 1}^{3}|\xi_j| \right)^2 -\left| \sum_{j = 1}^{3} \xi_j  \right|^2 }{  \sum_{j = 1}^{3}|\xi_j|   + \left| \sum_{j = 1}^{3} \xi_j  \right|   }
\\
& \leq 12\sigma \text{med}\left( |\xi_1|,  |\xi_2|, |\xi_3| \right).
\end{aligned}\end{equation}

By symmetry, we may assume $|\xi_1| \le |\xi_2|\le |\xi_3|$. 
Then by \eqref{qest}
\begin{align*}
|\widehat{ N_3(v)}(\xi)|& \leq 12 \sigma \int_{\xi = \xi_{1} + \xi_{2} +  \xi_{3}}   |\widehat v(\xi_{1}) |  \cdot  |\xi_2| |\widehat{v}(\xi_{2}) | \cdot  |\widehat{v}(\xi_{3}) |\, d\xi_1 d\xi_2 d\xi_3
\\
&=12 \sigma \int_{\xi = \xi_{1} + \xi_{2}+  \xi_{3}}    \widehat V(\xi_{1}) \cdot |\xi_2 | \widehat{V}(\xi_{2})\cdot  \widehat{V}(\xi_{3})  \, d\xi_1 d\xi_2 d\xi_3
\\
&=12 \sigma \mathcal F_x\left[ V \cdot  |D_x|V \cdot V \right](\xi).
\end{align*}
Therefore, using Plancherel, H\"{o}lder  and Sobolev inequalities we get
\begin{align*}
\|  N_3(v) \|_{L^{2}_x} & \le 12 \sigma \| V \cdot  |D_x|V \cdot V  \|_{L^{2}_{x}} 
\\
&\le 12 \sigma
\| V  \|_{L^{\infty}_{x}} \||D_x|V   \|_{L^{2}_{x}} \| V  \|_{L^{\infty}_{x}} 
\\
 & \lesssim  \sigma \|  V \|_{H^{2}}^3 \sim   \sigma \|  v \|_{H^{2}}^3
\end{align*}
as desired.

\end{document}